\documentclass[12pt,a4paper,oneside]{amsart}

\usepackage{amsfonts, amsmath, amssymb, amsthm, hyperref}
\usepackage{anysize}

\newtheorem{theorem}{Theorem}[section]
\newtheorem*{theorem*}{Theorem}
\newtheorem{lemma}[theorem]{Lemma}
\newtheorem{conjecture}[theorem]{Conjecture}
\newtheorem{corollary}[theorem]{Corollary}

\theoremstyle{definition}
\newtheorem{definition}[theorem]{Definition}

\theoremstyle{remark}
\newtheorem{remark}[theorem]{Remark}

\newtheorem{problem}[theorem]{Problem}

\numberwithin{equation}{section}

\DeclareMathOperator{\dist}{dist}
\DeclareMathOperator{\conv}{conv}

\newcommand{\const}{{\rm const}}

\renewcommand{\epsilon}{\varepsilon}
\renewcommand{\phi}{\varphi}
\renewcommand{\kappa}{\varkappa}

\newcommand*{\E}{\mathbb{E}}    
\newcommand*{\R}{\mathbb{R}}    
\newcommand*{\C}{\mathbb{C}}    
\renewcommand*{\S}{\mathbb{S}}   
\newcommand*{\Z}{\mathbb{Z}}    
\newcommand*{\T}{T}             
\newcommand*{\SO}{\mathrm{SO}}  

\newcommand*{\doi}[1]{\href{http://dx.doi.org/#1}{doi:#1}}

\begin{document}

\title{Suborbits in Knaster's problem}

\author{Boris~Bukh$^{*}$}
\email{bbukh@math.cmu.edu}
\address{Boris Bukh, Dept. of Math. Sciences, Wean Hall 6113, Carnegie Mellon University, Pittsburgh, PA 15213, USA}
\thanks{{$^{*}$} Work was partially done while the author was supported by the University of Cambridge, and Churchill College.}

\author{Roman~Karasev$^{**}$}
\email{r\_n\_karasev@mail.ru}
\address{Roman Karasev, Dept. of Mathematics, Moscow Institute of Physics and Technology, Institutskiy per. 9, Dolgoprudny, Russia 141700}
\address{Roman Karasev, Institute for Information Transmission Problems RAS, Bolshoy Karetny per. 19, Moscow, Russia 127994}
\address{Roman Karasev, Laboratory of Discrete and Computational Geometry, Yaroslavl' State University, Sovetskaya st. 14, Yaroslavl', Russia 150000}

\thanks{{$^{**}$} Supported by the Dynasty Foundation, the President's of Russian Federation grant MD-352.2012.1, and the Russian government project 11.G34.31.0053.}

\subjclass[2010]{46B07, 05D10, 52C99}
\keywords{Knaster's problem, Euclidean Ramsey theory}

\begin{abstract}
In this paper we exhibit a similarity between Euclidean Ramsey problems and Knaster-type problems. By borrowing ideas from Ramsey theory we prove weak Knaster properties of non-equatorial triangles in spheres, and of simplices in Euclidean spaces.
\end{abstract}

\maketitle

\section{Introduction}
The Borsuk--Ulam theorem asserts that for every continuous map $f\colon \S^d\to\R^d$
there is a pair of antipodal points $x,-x$ such that $f(x)=f(-x)$. Hopf~\cite{hopf_orig}
extended this, he showed that for any two points $x,y\in \S^d$ and any
continuous $f\colon \S^d\to \R^d$ there is a rotation $\sigma\in \SO(d+1)$
of the sphere such that $f(\sigma x)=f(\sigma y)$. Motivated by these
results, Knaster~\cite{knaster_orig} asked whether for every $\ell$-tuple
of points $X=\{x_1,\dotsc,x_\ell\}$ and every function $f\colon \S^{d+\ell-2}\to\R^d$
there is a set $Y=\{y_1,\dotsc,y_\ell\}$ that is isometric to $X$, and such that
$f(y_1)=\dotsb=f(y_\ell)$. The conjecture was shown to be false for $d\geq 2$
from the dimension considerations~\cite{makeev_counterexample,babenko_bogatyi,chen_counterexamples}, while the case of $d=1$ remained unsolved until a 
counterexample was found in~\cite{kashin_szarek} (see also~\cite[Miniature~32]{matousek_miniatures} for a simple exposition).
This suggests the following conjecture:
\begin{conjecture}[Weak Knaster conjecture]
There exists $n=n(\ell, d)$ such that for any $\ell$ points $X=\{x_1,\ldots,x_\ell\}$ on the unit sphere $\S^{n-1}$ and any continuous 
map $f\colon \S^{n-1}\to \R^d$ there exists a rotation $\rho\in \SO(n)$ such that
\[
f(\rho x_1) = f(\rho x_2) = \dots = f(\rho x_\ell).
\]
\end{conjecture}
The only settled cases of the conjecture are $n(2,d)=d+1$ (Hopf's theorem), and
$n(3,1)=3$ which is a result of Floyd~\cite{floyd}, but neither $n(4,1)$ nor $n(3,2)$ 
are known to be finite. 

An analogous phenomenon has been investigated in Euclidean Ramsey theory. A finite point set $X$ is \emph{Ramsey} if for every integer $t$ there is a $D=D(X,t)$ such that every 
$t$-coloring $\chi\colon \R^D\to [t]$ contains a rigid copy of $X$ on which $f$ is constant (or expressed in colorful language, a monochromatic copy of $X$). It is open which sets are Ramsey, but several classes of sets are known to be Ramsey: triangles~\cite{frankl_rodl_triangles_orig}, 
regular $m$-gons~\cite{kriz_permutation_groups}, non-degenerate simplices~\cite{frankl_rodl_simplices}, 
trapezoids~\cite{kriz_trapezoids}, and others. The most general result is that of K\v{r}\'{\i}\v{z}~\cite{kriz_permutation_groups} 
which asserts that any set $X$ that embeds into a finite set $Y$ with a 
solvable transitive symmetry group is Ramsey. In fact, the result of K\v{r}\'{\i}\v{z}
is slightly more general.
In a recent work of Leader--Russell--Walters~\cite{lrw_transitive} a very appealing conjecture 
is put forward, which would imply the same for all transitive symmetry groups, solvable or not.

Motivated by these results, we use the idea of embedding a set into an orbit of a group to
establish several positive Knaster-type results. A new obstacle compared to 
Euclidean Ramsey theory is that the only groups that we are currently able to handle are $p$-tori.
A \emph{$p$-torus} is a finite group isomorphic to $\Z_p^\alpha$ whose action on $\S^{n-1}$ is induced by a linear orthogonal action on the ambient $\R^n$. Our main topological tool is the following result:

\begin{theorem}
\label{knaster-orbits}
Suppose $X\subset \S^{k-1}$ is an orbit of an orthogonal  action of a $p$-torus, which is a group $G=(\Z_p)^\alpha$. Then for $n\ge d(|G|-1) + k$ and for every continuous map $f\colon \S^{n-1}\to\R^d$ there is an isometric copy of $X$ on which $f$ is constant.
\end{theorem}

We remark that an analogous result is true for $p$-tori replaced by general finite $p$-groups (see~\cite[Lemma 1]{dolnikov_karasev}), but without any explicit  bound on $n$. This extra strength does not appear to be useful in the results that follow.

A \emph{suborbit} of a group action is a subset of an orbit. In the view of Theorem~\ref{knaster-orbits}, to approach the weak Knaster conjecture we would 
want to know which sets are suborbits of $p$-tori.

\begin{definition}
A set $X\subset \S^{k-1}$ is called a \emph{spherical suborbit of a group $G$}, if $X$ is contained in an orbit of $G$ with respect to some orthogonal action of $G$ on a sphere of possibly larger dimension $\S^{n-1}$ containing an isometric copy of $\S^{k-1}$ in the standard way. If $G=\Z_p^\alpha$ is a $p$-torus, then we simply say that $X$ is \emph{spherical sub-$p$-toral}.
\end{definition}

One can weaken Knaster's conjecture even further by considering the maps not from $\S^{n-1}$, but from all of $\R^n$. This case is more analogous to the problems studied in Euclidean Ramsey theory (but see~\cite{graham_sphere} and~\cite{matousek_rodl_sphere}). In this setting we can consider more general suborbits:

\begin{definition}
A set $X\subset \R^k$ is called a \emph{Euclidean suborbit of a group $G$}, if $X$ is contained in an orbit of $G$ with respect to some action of $G$ on $\R^n$ by isometries for some $n\geq k$. If $G=\Z_p^\alpha$ is a $p$-torus, then we simply say that $X$ is \emph{Euclidean sub-$p$-toral}.
\end{definition}

Unfortunately, not every set is a suborbit. Leader--Russell--Walters~\cite{lrw_quadrilateral} showed the following negative result (see also~\cite{eberhard_hdsphere} for a higher-dimensional generalization):
\begin{theorem*}[Leader--Russell--Walters]
There exists a four-point set on $\S^1$ that is not a Euclidean suborbit of any finite group action.
\end{theorem*}

We adapt their argument to establish similar results for spherical suborbits, showing in effect that this approach will not solve the weak Knaster conjecture even for three points.
\begin{theorem}\label{thm_central_suborbits}
There exists a three-point set on $\S^1$ that is not a spherical suborbit of any finite group action.
\end{theorem}

This means that the suborbit approach to the weak Knaster conjecture fails in general. Still, some particular cases may deserve further investigation, which we do in the subsequent sections.

\section{Dvoretzky's theorem and suborbits}
\label{dvor-sec}

It is known (see~\cite{milman_local}) that the Knaster problem is related to the Dvoretzky theorem. The original Knaster conjecture would imply the Dvoretzky theorem with good estimates on the dimension of the almost spherical section. The weak Knaster conjecture could also give some reasonable results for the Dvoretzky theorem.

In view of Theorem~\ref{knaster-orbits}, it would be interesting to solve the following problem with good estimates on $n$ and $|G|$ in terms of $k$ and $\varepsilon$: 

\begin{problem}
\label{approx-p-suborbit}
For any $k$ and $\varepsilon > 0$ find an orthogonal action of a $p$-torus $G$ on $\S^{n-1}$ and a point $v\in \S^{n-1}$ such that the isometrically included $\S^{k-1}\subset \S^{n-1}$ is contained in the $\varepsilon$-neighborhood of the orbit $Gv$.
\end{problem}

We may rephrase it as follows: Can $\S^{k-1}$ be $\varepsilon$-approximated by a spherical sub-$p$-toral set for any $\varepsilon>0$?

Unfortunately, Problem~\ref{approx-p-suborbit} also has a negative solution. In fact, 
it is not clear how to deduce the Dvoretzky theorem from a positive solution to Problem~\ref{approx-p-suborbit} if it were true; in the standard proofs (see~\cite{milman_schechtman} or~\cite{ledoux_book}, for example) there 
remain some subtleties needed to control the Lipschitz constant through the estimation of the average norm. We also refer the reader to~\cite{burago_ivanov_tabachnikov}, where some other topological approaches to the Dvoretzky theorem are also shown to fail. 

To start investigating this problem, note that any orthogonal action of a $p$-torus (or any other abelian group) $G$ on $\R^n$ is an orthogonal sum of at most $2$-dimensional irreducible representations. One simple way to see this is to observe that this action is diagonalizable after complexification to $\C^n$ and then produce an at most two-dimensional $G$-invariant subspace in $\R^n$ for every eigenvector of this action in $\C^n$. 

The two-dimensional irreducible representations of $G$ only occur when $p$ is odd and then the group $G$ keeps orientations on them. If some of the $G$-invariant subspaces have dimension $1$ (and $p=2$) then we $\oplus$-add the same representation to $\R^n$ (since we are free to increase the dimension in this problem) and again obtain a two-dimensional $G$-invariant subspace with action of $G$ keeping the orientation. Finally, the whole $\R^n$ can be assumed to be split into two-dimensional $G$-invariant subspaces with $G$-action keeping the orientation, that is the action of $G$ can be extended to the action of the torus $\T^n$, which rotates every two-dimensional $G$-invariant subspace independently.

This means that in Problem~\ref{approx-p-suborbit} it is sufficient to consider the standard action of the (honest) torus $\T^n$, in place of $G$, on $\C^n\simeq \R^{2n}$ and its corresponding unit sphere $\S^{2n-1}$. After these preparations we observe the following:

\begin{theorem}
\label{suborbit3-failure}
There exists $\varepsilon > 0$ with the following property: For any $3$-dimensional linear subspace $L\subseteq \R^{2n}$ its sphere of unit vectors $S(L)$ is 
not contained in an $\varepsilon$-neighborhood of any orbit $\T^nx\subset \R^{2n}$. Consequently Problem~\ref{approx-p-suborbit} has a negative solution for $k \ge 3$.
\end{theorem}

\begin{lemma}
Suppose $r$ and $c$ are non-negative real numbers, and $t$ is a random variable uniformly
distributed on the interval $[-1,1]$. Then 
$\E\bigl[ ( r\sqrt{1-t^2}-c)^2\bigr]\geq \tfrac{1}{16} \E\bigl[ (r\sqrt{1-t^2})^2 ]$.
\end{lemma}
\begin{proof}
As $\int_0^1 (1-t^2)\,dt=2/3$ it suffices to show that $\int_0^1 (\sqrt{1-t^2}-c)^2\,dt\geq 1/24$ for every $c\geq 0$. The
integral is $\tfrac{2}{3}+c^2-c\pi/2$, which is minimized at $c=\pi/4$. The minimum is $2/3-\pi^2/16>1/24$.
\end{proof}

\begin{proof}[Proof of theorem~\ref{suborbit3-failure}]
It is natural to identify $\R^{2n}$ with $\C^n$. Then the action of $\T^n$ will be by multiplying all the coordinates 
by complex numbers of unit norm. So an orbit is defined by the relations
$$
|z_1|= c_1, |z_2|=c_2, \ldots, |z_n|=c_n,
$$
and the squared distance to an orbit is given by:
\begin{equation}
\dist(\bar z, \T^n\bar c)^2 = \sum_{i=1}^n \left(|z_i| - c_i\right)^2.
\end{equation}

Assume that $L$ is an image of an $\R$-linear map $\lambda\colon \R^3\to \C^n$ 
with coordinates $\lambda_i\colon \R^3\to \C$. The map $\lambda_i$
has a non-trivial kernel. Let $e_i\in \R^3$ be a unit vector such that
$\lambda_i(e_i)=0$. 

Our next step is to show that if $v$ is a vector chosen uniformly on $\S^2$, then $\lambda(v)$ is
far from $T^n \bar c$. For every $i$ the vector $v$ can be represented as follows.
First pick a vector $u_i$ uniformly on the equatorial circle perpendicular to $e_i$, and then pick a real 
number $t_i$ uniformly at random from $[-1,1]$. Then set $v=t_i\cdot e_i+\sqrt{1-t_i^2}\cdot u_i$. By the theorem of Archimedes~\cite{archimedes_sphere} on the surface area of sections of a sphere by parallel planes, the resulting vector $v$ is uniformly distributed on $\S^2$. 

We shall think of the random variables $t_i$ and $u_i$ as defined on the same probability space, coupled by the common
value of $u$. From linearity of expectation and by the preceding lemma we deduce
\begin{align*}
  \E_v \bigl[ \dist(\lambda(v),\T^n\bar c)^2 ]&=\sum_{i=1}^n \E_v \left(\lvert\lambda_i(v)\rvert - c_i\right)^2\\
&=\sum_{i=1}^n 
\E_{u_i} \E_{t_i} \bigl[(\lvert\lambda_i(u_i)\rvert\sqrt{1-t_i^2}-c_i)^2\bigr]\\
&\geq \sum_{i=1}^n \E_{u_i} \E_{t_i} \bigl[(\lvert\lambda_i(u_i)\rvert\sqrt{1-t_i^2})^2\bigr]/16\\
&= \sum_{i=1}^n \E_v |\lambda_i(v)|^2/16=1/16\\
\end{align*}
since $\lambda(v)$ has norm~$1$. Hence, no matter what
$\bar c$ is, there is a point in the image of $\lambda$
that is at distance at least $1/4$ from $\T^n\bar c$.
\end{proof}

\begin{remark}
It follows that sufficiently dense subsets of $\S^2$ are not spherical suborbits of commutative group actions and cannot be approximated by spherical suborbits of commutative group actions. Indeed, if a finite set $X\subset \S^2$ is such that $\S^2$ is in the $\varepsilon$-neighborhood of $X$ and $X$ is in the $\varepsilon$-neighborhood of an orbit $Gv$ (for an abelian $G$ after some isometric inclusion $\S^2\subset \R^N$) then $\S^2$ is in the $2\varepsilon$-neighborhood of this orbit, which contradicts Theorem~\ref{suborbit3-failure}.
\end{remark}

\begin{remark}
In~\cite[Proposition~2.3]{matousek_rodl_sphere} it is proved that any $\S^k$ \emph{can} be $\varepsilon$-approximated by a suborbit of the standard coordinate-wise action of the symmetric group $\Sigma_n$ on $\S^{n-1}$ with sufficiently large $n = n(k,\varepsilon)$. 
\end{remark}

\begin{problem}
What about suborbits of $p$-groups in Problem~\ref{approx-p-suborbit}? Can the proof in~\cite{matousek_rodl_sphere} be modified to produce a $p$-group instead of $\Sigma_n$? Note that in~\cite{dolnikov_karasev} the following result was 
established during the proof of~\cite[Theorem~1]{dolnikov_karasev}: There exists $n=2^\ell$ depending on $k$ and $d$, and a $k$-dimensional subspace $L\subseteq \R^n$ with the following property: any $G$-invariant polynomial $P\colon \R^n\to \R$ of degree at most $d$ is constant on the unit sphere of $L$. Here $G$ is $\Z_2^n\rtimes \Sigma_n^{(2)}$, the $2$-group generated by reflections and $2$-Sylow permutations 
of coordinates. Does it follow that such $S(L)$ is close to a $G$-orbit in $\R^n$?
\end{problem}

\begin{remark}
Note that orbits of nonabelian $p$-groups are not so useful in the Dvoretzky theorem as orbits of $p$-tori because in the corresponding version of Theorem~\ref{knaster-orbits} 
no explicit formula for $n$ is known so far, see the discussion in~\cite{dolnikov_karasev}.
\end{remark}

\section{Euclidean suborbits and the Euclidean Knaster problem}
\label{pos-eu-sec}
In this section we consider Euclidean suborbits, as in the Ramsey-type problems. Informally, their difference from spherical suborbits is that, after possibly increasing the dimension, we are free to choose any origin of the group action. The following argument is a minor adaptation of Frankl--R\"{o}dl's argument from~\cite{frankl_rodl_triangles_orig,frankl_rodl_simplices}.

\begin{lemma}
\label{sub-p-prod}
If $X$ and $Y$ are Euclidean sub-$p$-toral, then so is $X\times Y$.
\end{lemma}

\begin{proof}
If $X$ is an orbit of a $p$-torus $G_1$ in its representation $V_1$ and $Y$ is an orbit of a $p$-torus $G_2$ in its representation $V_2$, 
then there is an obvious componentwise $G_1\times G_2$-action on $V_1\times V_2$ with orbit $X\times Y$. Passing to suborbits is obvious.
\end{proof}

\begin{lemma}
\label{sub-p-twopoint-prod}
If $X$ is Euclidean sub-$p$-toral, then so is $X\times \{0,t\}$.
\end{lemma}

\begin{proof}
Let $Y$ be the regular $(p-1)$-simplex of side length $t$; it obviously has an orthogonal $\Z_p$-action by cyclic permutations. 
Then $X\times Y$ is Euclidean sub-$p$-toral by Lemma~\ref{sub-p-prod} and contains $X\times \{0,t\}$.
\end{proof}

\begin{lemma}
\label{wide-p-triangles}
There is a sequence of angles $\alpha$ tending to $\pi$ such that the isosceles triangle 
with central angle $\alpha$ is Euclidean sub-$p$-toral. Here, $p$ tends to infinity as $\alpha$ tends to $\pi$.
\end{lemma}

\begin{proof}
Consider $p$ points uniformly distributed on a circle. Then three consecutive points form a triangle with the central angle $\pi(1-\frac{1}{p})$.
\end{proof}

\begin{lemma}
\label{all-iso-p-triangles}
For every angle $\alpha<\pi$, there is a Euclidean $p$-toral set $X$ containing an isosceles triangle with central angle $\alpha$. Here, $p$ 
can be chosen to be any prime larger than a certain $p_0(\alpha)$.
\end{lemma}

\begin{proof}
Let $\alpha'>\alpha$ be any of the special angles in Lemma~\ref{wide-p-triangles}. Let $ABC$ be an isosceles triangle with angle $\alpha'$. Let the base be $BC$. Consider the Cartesian product of $ABC$ with $\{0,t\}$. The three points $(A,t)$, $(B,0)$, $(C,0)$ form an isosceles triangle whose angle varies from $\alpha'$ to $0$ as $t$ varies from $0$ to $+\infty$. Hence, there is a $t$ such that the angle is $\alpha$. The result now follows from Lemma~\ref{sub-p-twopoint-prod}.
\end{proof}

\begin{theorem}
\label{all-p-triangles}
Every triangle is Euclidean sub-$p$-toral for all sufficiently large primes $p$.
\end{theorem}

\begin{proof}
Consider an isosceles triangle $ABC$ with base $BC$ from Lemma~\ref{all-iso-p-triangles}. Consider 
the product of $ABC$ with $\{0,t\}$.  The three points $(A,0)$, $(B,t)$, $(C,0)$ form any 
triangle one desires, see the details in~\cite{frankl_rodl_triangles_orig}.
\end{proof}

The proof that all simplices are Euclidean sub-$p$-toral is similar, if one follows the idea from the follow-up paper of Frankl--R\"{o}dl~\cite{frankl_rodl_simplices}:

\begin{lemma}
\label{z-sub-p-toral}
For $\varepsilon>0$, a positive integer $m$, and a sufficiently large prime $p\ge p_0(\varepsilon, m)$ there is a Euclidean sub-$p$-toral set $B=\{b_1,...,b_m\}$ 
such that $|\dist(b_i, b_j)-|i-j||<\varepsilon$ for every pair $\{i,j\}$.
\end{lemma}

\begin{proof}
We may take $B$ to be $m$ consecutive points of a regular $p$-gon.
\end{proof}

\begin{theorem}
\label{sup-p-dense}
For every $\delta>0$, every set $A= \{a_1,...,a_m\}\subset \R^n$, and sufficiently large prime $p\ge p_0(A, \delta)$ there is a Euclidean sub-$p$-toral set $S = \{s_1,...,s_m\}$ in $\R^N$ of possibly larger dimension such that $\dist(a_i, s_i)< \delta$ for all $i$.
\end{theorem}

\begin{proof}
Without loss of generality assume $A\subset [0,1]^n$. Let $s$ be large (in terms of $1/\delta$). Consider the grid $G(s)=\{0,1/s,2/s,...,1\}^n\subset [0,1]^n$. Let $B'=B/s$ where the set $B$ with $|B|=s+1$ is as in Lemma~\ref{z-sub-p-toral}. The set $B'$ approximates $\{0,1/s,2/s,...,1\}$ and its Cartesian power $G'(s) = (B')^n$ approximates the grid $G(s)$,
and is sub-$p$-toral by Lemma~\ref{sub-p-prod}.

Now we may approximate $A$ by a subset $A'$ of $G(s)$ and then approximate $A'$ by the corresponding subset $S$ of the deformed grid $G'(s)$.
\end{proof}

Note that Theorem~\ref{sup-p-dense} is analogous to~\cite[Corollary~4.2]{frankl_rodl_simplices}. 
There is an analogue of~\cite[Corollary~3.2]{frankl_rodl_simplices}:

\begin{lemma}
\label{almost-regular-p-toral}
Let $S=\{s_0,\ldots, s_n\}$ be a regular simplex in $\R^n$. There exists $\delta>0$ with the following property: Any subset $A=\{a_0,\ldots, a_n\}\subset \R^n$ such that $\dist(a_i, s_i)<\delta$ for any $i$ is Euclidean sub-$p$-toral for primes $p\ge 2$.
\end{lemma}

\begin{proof}
By~\cite[Lemma~3.1]{frankl_rodl_simplices}, for small enough $\delta$, any $A$ that is $\delta$-close to the regular simplex $S$ is isometric to a subset of vertices of certain ``brick'' $\{0, t_1\}\times \{0, t_2\}\times \dots \{0, t_N\}$ with $N = \binom{n}{2}$. A brick is Euclidean sub-$p$-toral for every $p\ge 2$ by Lemma~\ref{sub-p-twopoint-prod}.
\end{proof}

Finally, like in~\cite[Theorem~5.1]{frankl_rodl_simplices} we conclude:

\begin{theorem}
\label{s-sub-p-toral}
Any affinely independent set $A\subset \R^n$ is Euclidean sub-$p$-toral for sufficiently large $p\ge p_0(A)$.
\end{theorem}

\begin{proof}
Following~\cite[Theorem~5.1]{frankl_rodl_simplices}, we observe that $A$ can be isometrically embedded into a product of an almost regular simplex $S_1$ and a simplex $S_2$ that approximates $A$ by Lemma~\ref{sup-p-dense}.
\end{proof}

Also, note that the result of Theorem~\ref{s-sub-p-toral} is almost in contradiction with Theorem~\ref{suborbit3-failure}. 
It is a consequence of the difference between Euclidean and spherical suborbits. Still, we are able to produce a Knaster-like 
consequence of Theorem~\ref{s-sub-p-toral}:

\begin{theorem}
\label{eu-weak-knaster}
Assume $X$ is an affinely independent subset of $\R^k$ and $d$ is a positive integer. Then there exists $n = n(X, d)$ such that for any continuous map $f\colon \R^n\to \R^d$ there exists an isometric image $X'\subset\R^n$ of $X$ such that $f|_{X'} = \const$.
\end{theorem}

\begin{proof}
By Theorem~\ref{s-sub-p-toral}, after increasing the dimension $k$ and the set $X$ we may assume that $X\subset \R^k$ is an orbit of a linear $G$-action on $\R^k$, where $G$ is a $p$-torus with $|G|=|X|=q=p^\alpha$. 

Now we can embed $\R^k$ into $\R^n$ so that $n-k\ge d(q-1)$ and let $G$ act on the complement $\R^{n-k}$ trivially. Obviously, the $G$-orbit $X$ must lie on a sphere $\S^{n-1}$ centered at the origin and we may apply Theorem~\ref{knaster-orbits} to $X$ and $f|_{\S^{n-1}}$. For completeness, we outline the rest of the proof below, which is also a proof for Theorem~\ref{knaster-orbits}.

For a given $\R^n$ with a continuous map $f:\R^n\to \R^d$, consider all possible isometric images of the set $X$ in $\R^n$ with mass center (and the center of the action of $G$) at the origin. The configuration space $\mathcal C$ of all possible images $X'=\phi(X)$ is isomorphic to the Stiefel manifold $V_{n, k}$, that is the space of all orthonormal $k$-frames in $\mathbb R^n$, since every isometric embedding of $X$ can be uniquely extended to an isometric embedding of its linear span. This space  $\mathcal C$ has the obvious action of $G$ by permuting the points of $X$, that is composing the isometry $\phi\in\mathcal C$ with $g^{-1}$ from the right for any $g\in G$. 

The crucial thing about the Stiefel manifold $V_{n,k}$ is that it is $(n-k-1)$-connected, indeed, adding the frame vectors one by one we represent it as an iterated bundle with all fibers spheres of dimension at least $n-k$. Now construct the $G$-equivariant map
$$
\tilde f\colon \mathcal C \mapsto (\R^d)^q
$$
by sending $X' = \phi(X)\in \mathcal C$ to the $q$-tuple $\{f(\phi(x))\}_{x\in X}$. 

Using the standard Borsuk--Ulam-type argument for a $p$-torus $G$ (see~\cite{volovikov_stiefel}, for example) and the connectivity of $\mathcal C$, we conclude that for $n-k\ge d(q-1)$ $\tilde f$ must map some $X'\in \mathcal C$ to the $G$-invariant part of the representation $(\R^d)^q$, which is the $d$-fold direct sum of the group algebra $\R[G]$ with itself. The invariant part of the group algebra $\R[G]$ is one-dimensional and consists of constant functions on $G$. So we obtain that the map $f$ is constant on some $X'$.
\end{proof}

\section{Three-point suborbits on the sphere}
\label{pos-sph-sec}
In this section we return to spherical $p$-toral suborbits and study three-point subsets $X\subset \S^2$. 

\begin{theorem} Let $X$ be a three-point subset of $\S^2$. 
\label{three-sph-sub-p}
\begin{enumerate}
\item\label{three-sph-parti} If the affine span of $X$ does not contain the origin then $X$ is spherical sub-$p$-toral for all sufficiently large primes $p$.
\item\label{three-sph-partii}  If $X\subset \S^1$ then $X$ is spherical sub-$p$-toral if and only if either $X$ is a right-angled triangle, or $X$ is a suborbit
of the standard $\Z_p$ action on the circle for some $p$.
\end{enumerate}
\end{theorem}

From this theorem and Theorem~\ref{knaster-orbits} we immediately obtain the following Knaster-type result:
\begin{corollary}
For any three-point set $X\subset \S^2$, whose affine span does not contain the origin, and a positive integer $d$ there exists $N = n(X) + d(q(X) - 1)$ with the 
following property: If $f\colon \S^{N-1}\to \R^d$ is a continuous map then for some isometric copy $X'\subset \S^{N-1}$ of $X$ the map $f$ is constant on $X'$.
\end{corollary}

\begin{proof}[Proof of Theorem~\ref{three-sph-sub-p}]
For the purpose of this proof a ``triangle'' is just a triple of points. We allow some of the points to be equal (these are degenerate triangles). 
To each triangle $ABC$ we associate a triple $(X,Y,Z)$ where $X=\dist(A,B)^2$, $Y=\dist(A,C)^2$, and $Z=\dist(B,C)^2$.
Let $\Sigma$ be the set of the triples that correspond to the triangles with circumradius equal to~$1$. Let $R=\conv(\{0\}\cup \Sigma)$. 
The triples in $R$ correspond to the triangles with circumradius at most~$1$. We shall abuse the notation and identify elements
of $R$ with corresponding triangles. 

The topological boundary of $R$ consists of a union of $\Sigma$ and three line segments $[0,(4,4,0)]$, $[0, (4,0,4)]$ and $[0, (0,4,4)]$. 
Indeed, only the triangles of circumradius $1$ and the degenerate triangles are on the boundary of $R$. 
Let $\Sigma_p\subset \Sigma$ be the set of triangles that are suborbits of the standard action of $\Z_p$
on $\S^1$. Note that $\Sigma_2$ includes the degenerate triangles $(4,4,0)$, $(4,0,4)$ and $(0,4,4)$ 
that correspond to pairs of antipodal points, but $0\not\in \Sigma_p$ because $0\not\in \Sigma$.
Let $R_p=\conv (\{0\}\cup \Sigma_p)$. We claim that all the triangles in $R_p$ are spherical sub-$p$-toral.

It suffices to prove that the set of spherical sub-$p$-toral triangles is a convex subset of $R$.
Suppose $(A_1,B_1,C_1)\subset \S^{k_1-1}$ and $(A_2,B_2,C_2)\subset \S^{k_2-1}$ are suborbits of $\Z_p^{\alpha_1}$ and $\Z_p^{\alpha_2}$ respectively. Suppose further that $\lambda_1,\lambda_2\in\R$ satisfy $\lambda_1^2+\lambda_2^2=1$. Let 
\[
   A=\lambda_1 A_1\oplus \lambda_2 A_2,\quad  B=\lambda_1 B_1\oplus \lambda_2 B_2,\quad C=\lambda_1 C_1\oplus \lambda_2 C_2. 
\]
The triangle $(A,B,C)$ lies on the sphere $\S^{k_1+k_2-1}$, and is a suborbit of the natural action of $\Z_p^{\alpha_1}\times \Z_p^{\alpha_2}$.
We see that the triple of squares of the side-lengths $(|AB|^2,|AC|^2,|BC|^2)$ is a convex
combination of $(|A_1B_1|^2,|A_1C_1|^2,|B_1C_1|^2)$  and $(|A_2B_2|^2,|A_2C_2|^2,|B_2C_2|^2)$ with coefficients $\lambda_1^2$ and $\lambda_2^2$. As $\lambda_1,\lambda_2$ were arbitrary,
it follows that the set of spherical sub-$p$-toral triangles is indeed a convex subset of~$R$.
We note that the same argument shows that $R$ itself is convex.

We are ready to prove part (\ref{three-sph-parti}) of the theorem. Let $P=(X,Y,Z)\in R$ be an interior point of $R$. We claim that for all sufficiently large $p$, the point $P$ is
in $\conv(\{0\}\cup \Sigma_p)$. Indeed, we can pick four rays emanating from $P$, not
all four in the same halfspace. Let $P_1,\dotsc,P_4$ be the intersections of these rays
with~$\Sigma$. Then $P$ is contained in the interior of $\conv \{P_1,\dotsc,P_4\}$.
As the points of $\Sigma_p$ are getting denser and denser in $\Sigma$, if $p$ is large,
there are points $P_1',\dotsc,P_4'\in \Sigma_p$ approximating $P_1,\dotsc,P_4\in \Sigma$ such that $P\in\conv\{P_1',\dotsc,P_4'\}$.
As every non-degenerate triangle of circumradius strictly less than $1$ is in the interior of $R$, part (\ref{three-sph-parti}) follows.

The proof of part (\ref{three-sph-partii}) relies on the reversal of the argument above. Every action of $\Z_p^n$ decomposes into the product of two-dimensional actions as in proof of Theorem~\ref{suborbit3-failure}.
Therefore, every orbit of $\Z_p^n$ is of the form $a_1 S_1\oplus \dotsb \oplus a_n S_n$
where $S_1,\dotsc,S_n$ are regular $p$-gons, and $a_1^2+\dotsb+a_n^2=1$. If
$A=(A_1,A_2,\dotsc,A_n)$, $B=(B_1,B_2,\dotsc,B_n)$, $C=(C_1,C_2,\dotsc,C_n)$ are three
points in such an orbit, then $(|AB|^2,|AC|^2,|BC|^2)$ is the convex combination of 
$\bigl\{(|A_iB_i|^2,|A_iC_i|^2,|B_iC_i|^2)\bigr\}_{i=1}^n$. In other words,
$R_p=\conv(\{0\}\cup\Sigma_p)$ is precisely the set of all spherical sub-$p$-toral sets.

It thus remains to explicitly describe the set $R_p\cap \Sigma$. 
This step hinges on an explicit formula for $\Sigma$. It is easy
to show from the formula for the circumradius of a triangle, that all the points $(X,Y,Z)$ of $\Sigma$
satisfy the relation
\begin{equation}
\label{square-sidelength}
XYZ + X^2+Y^2+Z^2 = 2XY + 2YZ + 2ZX.
\end{equation}
A point $P$ may belong to $R_p\cap \Sigma$ only if $P\in \Sigma_p$ or $P$ is in a convex hull of several points of $\Sigma$ distinct from $P$.
Since $R$ is convex, the latter might happen only if $P$ lies on a line segment that is wholly contained in $\Sigma$.
Let $V$ be the algebraic variety defined by the equation~\eqref{square-sidelength}. As $\Sigma\subset V$, a line segment on $\Sigma$
belongs to a line on $V$. As $XYZ$ is the only cubic term of \eqref{square-sidelength} it follows that every line on $V$ is parallel to one
of the coordinate planes. Pick such a line, and assume without loss of generality that $X=X_0$ on the line. Then \eqref{square-sidelength}
after some transformations reads:
\[
(Y-2)^2 + (Z-2)^2 + (X_0-2)(Y-2)(Z-2) + (X_0-2)^2-4 = 0,
\]
and we conclude that the only cases when a straight line satisfies this is $X_0=0$ and $Y=Z$, or $X_0=4$ and $Y+Z=4$. 
The former case corresponds to the degenerate triangles consisting of two points. The latter case corresponds to
the right-angled triangles inscribed into $\S^1$. As the right-angled triangles are suborbits of a $\Z_2^2$ action,
the proof of (\ref{three-sph-partii}) is complete.
\end{proof}

\begin{remark}
Certain steps of the previous proof can be repeated for the general case of $n$-element sub-$p$-toral sets. They decompose as ``direct sums'' of two-dimensional $\Z_p$ orbits. All $n$-element subsets of unit spheres are naturally parameterized by the $n\times n$ positive semidefinite symmetric matrices with $1$'s on the diagonal, that is Gram matrices. Call this convex set of matrices $R^n$. Subsets of $\S^1$ then correspond to the set $\Sigma^n\subset R^n$ of positive semidefinite matrices of rank at most $2$. Again, for every prime $p$ there is a finite set $\Sigma^n_p$ corresponding to suborbits of $\Z_p$ on the circle $\S^1$, possibly degenerate. 

We observe that the sub-$p$-toral $n$-tuples correspond to $\conv \Sigma^n_p$. For $n>3$, the set $\Sigma^n$ is no more equal to ``almost the boundary'' of $R^n$ 
and even the set $\conv \Sigma^n$ may not coincide with the whole $R^n$. In fact, Theorem~\ref{suborbit3-failure} implies that for large $n$ the set
$\conv \Sigma^n$ will not coincide with $R^n$.
\end{remark}

Finally, we prove Theorem~\ref{thm_central_suborbits} about three-point suborbits of arbitrary groups.
\begin{proof}[Proof of Theorem~\ref{thm_central_suborbits}]
Say that a triple of distinct points $A,B,C\in \S^1$ has parameters $(a,b,c)$ if the linear relation 
$$
aA+bB+cC=0
$$
holds true. These parameters are defined up to rescaling. Suppose $B=gA$ and $C=hA$ where $g,h$ are elements of a finite group 
acting on an ambient $\S^N\supseteq \S^1$. Since every representation of a finite group is equivalent to the matrix representation 
with algebraic entries, it follows that $B=GA$ and $C=HA$ for some algebraic matrices $G$ and $H$. Then $(aI+bG+cH)A=0$, which 
implies that $\det(aI+bG+cH)=0$. This is a polynomial equation in $a,b,c$, evidently nondegenerate. The set 
of triples $(a,b,c)$ such that there is a triple $A,B,C\in \S^1$ with parameters $(a,b,c)$ is $3$-dimensional, 
whereas this polynomial defines a surface. Since there are only countably many possible groups of symmetry and 
their representations, it follows that almost every triple in $\S^1$ is not a suborbit of a finite group action on an ambient $\S^N\supseteq \S^1$. 
\end{proof}

\medskip
{\bf Acknowledgment.} We thank the anonymous referee for valuable remarks.

\bibliographystyle{plain}
\bibliography{suborbits}

\end{document}